\documentclass[11pt]{amsart}%
\usepackage{amscd,amsmath,latexsym,amsthm,amsfonts,amssymb,graphicx,color,geometry}
\usepackage{enumerate}
\usepackage[pdftex,plainpages=false,pdfpagelabels,
colorlinks=true,citecolor=blue,linkcolor=black,urlcolor=black,
filecolor=black,bookmarksopen=true,unicode=true]{hyperref}
\usepackage[norefs,nocites]{refcheck}
\usepackage{amsmath}
\usepackage{amsfonts}
\usepackage{amssymb}
\usepackage[T1]{fontenc}
\usepackage{graphicx}
\usepackage{xcolor}
\usepackage{colortbl}
\usepackage{float}
\usepackage{framed}
\usepackage{pgf,tikz,pgfplots}
\usepackage{tkz-base}
\usepackage{multicol}
\usepackage{tabularx}
\usepackage{tkz-fct}%
\providecommand{\U}[1]{\protect\rule{.1in}{.1in}}
\tikzset{>=Triangle}
\newtheorem{theorem}{Theorem}[section]
\newtheorem{proposition}[theorem]{Proposition}

\newtheorem{corollary}[theorem]{Corollary}

\newtheorem{lemma}[theorem]{Lemma}

\newtheorem{definition}[theorem]{Definition}
\numberwithin{equation}{section}
\pgfplotsset{compat=1.17}
\begin{document}
\title[Complements of Non-Minimal Subspaces: Characterization Results]{Complements of Non-Minimal Subspaces: Characterization Results}
\author[G. Ribeiro]{Geivison Ribeiro}
\address{Departamento de Matem\'{a}tica \\
Universidade Federal da Para\'{\i}ba \\
58.051-900 - Jo\~{a}o Pessoa, Brazil.}
\email{geivison.ribeiro@academico.ufpb.br}
\thanks{G. Ribeiro is supported by Grant 2022/1962, Para\'{\i}ba State Research
Foundation (FAPESQ)}
\subjclass[2020]{15A03, 46B87, 46A16}
\keywords{Lineability, Spaceability, $F$-spaces, Banach Spaces, Basic Sequences }

\begin{abstract}
Inspired by the work of L. Drewnowski in [Studia Math. \textbf{77} (1984)
373--391], our research reveals new insights and characterizes the notion of
spaceability in the context of complements of subspaces (not necessarily
closed) within the universe of $F$-spaces in terms of $\left[  \mathcal{S}%
\right]  $-lineability.

\end{abstract}
\maketitle

\section{Introduction and Background}

The investigation into the existence of infinite-dimensional closed subspaces
devoid of a \emph{basic sequence} within the context of $F$-spaces made
significant strides with the work of N. Kalton in \cite{Kalton}, who
demonstrated the following result:

\begin{theorem}
\cite[Theorem 1.1]{Kalton}There is a quasi-Banach space $E$ with a
one-dimensional subspace $L$ so that

\begin{enumerate}
\item[$\left(  1\right)  $] if $E_{0}$ is a closed infinite-dimensional
subspace of $E$ then $L\subset E_{0}$
\end{enumerate}

and

\begin{enumerate}
\item[$\left(  2\right)  $] $E/L$ is isomorphic to the Banach space $\ell_{1}$.
\end{enumerate}

In particular, $E$ contains no basic sequence and is minimal.
\end{theorem}

In addition to revealing that the set $(E\setminus L)\cup\left\{  0\right\}  $
does not contain an infinite-dimensional closed subspace, this result is
particularly significant when contrasted with the classical result proved by
Enflo \cite{Enflo} and provided by Banach \cite[p. 238]{Banach}, which states
that \emph{every Banach space contains an infinite-dimensional closed subspace
with a basic sequence}. While the classical theorem guarantees the presence of
such subspaces in Banach spaces, Kalton's result reveals the additional
complexity and richness of linear structures in $F$-spaces.

This distinction not only underscores the fundamental difference between
$F$-spaces and Banach spaces but also highlights the diversity of possible
linear configurations in $F$-space environments.

According to \cite{KShapiro}, an $F$-space $E$ contains a basic sequence if
and only if there exists a strictly weaker Hausdorff vector topology on $E$.
Thus, the existence of a space without a basic sequence is equivalent to the
existence of a (topologically) minimal space (i.e., one in which no strictly
weaker Hausdorff vector topology exists).

Among the positive results regarding the existence of closed subspaces
containing basic sequences, one notable contribution is the work by Bastero in
\cite{Bastero}, which provides information that every subspace of $L_{p}[0,1]$
with $0<p<1$, contains a basic sequence, and a result due to L. Drewnowski in
\cite{Drew}, as follows:

\begin{theorem}
\label{TheoDrew copy(1)}$\left(  \text{\cite[Theorem 3.3]{Drew}}\right)  $For
an $F$-space $E$, if $F$ is a closed subspace non-minimal of $E$ and
$\dim\left(  E/F\right)  =\infty$, then there exists a closed non-minimal
subspace $G$ of $E$ such that $F\cap G=\left\{  0\right\}  $ and $F+G$ is not closed.
\end{theorem}

To provide context, let us fix some notations and terminology. We denote
$\mathbb{N}$ as the set of positive integers, $\mathbb{R}$ as the real scalar
field, and $\mathbb{C}$ as the complex scalar field. Unless explicitly stated
otherwise, all linear spaces are over $\mathbb{K}=\mathbb{R}\text{ or
}\mathbb{C}$. The term subspace will be used instead of vector subspace.
Additionally, $\alpha$ and $\beta$ represent cardinal numbers, $\text{card}%
(A)$ denotes the cardinality of set $A$, $\aleph_{0}:=\text{card}(\mathbb{N}%
)$, and $\mathfrak{c}:=\text{card}(\mathbb{R})$. Furthermore, if $E$ is a
vector space and $W$ is a subspace of $E$, we denote the algebraic codimension
of $W$ by $\dim(E/W)$.

In terms of notions, assume that $E$ is a vector space and let $\beta\leq
\dim(E)$ be a cardinal number. A subset $A \subset E$, is said to be:

\begin{itemize}
\item \textit{lineable} if there is an infinite dimensional subspace $F$ of
$E$ such that $F\smallsetminus\{0\}\subset A$, and

\item \textit{$\beta$-lineable} if there exists a subspace $F_{\beta}$ of $E$
with $\dim(F_{\beta})=\beta$ and $F_{\beta}\smallsetminus\{0\}\subset A$ (in
this context, a set is \emph{lineable} if it is $\aleph_{0}$-lineable).
\end{itemize}

This concepts was first introduced in the seminal paper \cite{AGSS} by Aron,
Gurariy and Seoane-Sep\'{u}lveda and, later, in \cite{Enflo, Gurariy-Quarta}.
Its essence is to investigate linear structures within exotic settings. A more
detailed description of the concept of lineability/spaceability can also be
found in \cite{Aron, Bernal, Botelho, Cariello, Enflo, Fonf}.

Moreover, if we consider a another cardinal number $\alpha$, with $\alpha
\leq\beta$, then $A$ is said to be:

\begin{itemize}
\item \textit{$(\alpha,\beta)$-lineable} if it is $\alpha$-lineable and for
every subspace $F_{\alpha}\subset E$ with $F_{\alpha}\subset A\cup\{0\}$ and
$\dim(F_{\alpha})=\alpha$, there is a subspace $F_{\beta}\subset E$ with
$\dim(F_{\beta})=\beta$ and
\begin{equation}
\label{ab}F_{\alpha}\subset F_{\beta}\subset A\cup\{0\}
\end{equation}
(hence $(0,\beta)$-lineability $\Leftrightarrow\ \beta$-lineability).
\end{itemize}

If $E$ is, in addition, is a topological vector space, then $A$ is called:

\begin{itemize}
\item \textit{$\beta$-spaceable} if $A\cup\{0\}$ contains a closed $\beta
$-dimensional subspace of $E$.

\item Moreover, if the subspace $F_{\beta}$ satisfying \eqref{ab} can always
be chosen closed, we say that $A$ is \textit{$(\alpha,\beta)$-spaceable} (see
\cite{FPT}).
\end{itemize}

In the quest for a deeper understanding of dimensional relationships within
spaceable sets, F\'{a}varo, Pellegrino, and Tomaz introduced this concept in
2019 (referenced as \cite{FPT}). This heightened conceptualization, although
rooted in conventional spaceability, demands meticulous consideration,
encompassing both geometric and topological properties within the set, as well
as relationships between the dimensions of subspaces therein. It is worth
noting that while conventional spaceability implies the variant known as
$\left(  \alpha,\alpha\right)  $-spaceability, the reverse assertion is not
always valid, as demonstred by Ara\'{u}jo et al$.$ in \cite{Araujo}. For a
more comprehensive exposition on this notion, we refer the reader to
\cite{Mikaela, Nacib, Gustavo, BRSH, Sheldon, Diogo/Anselmo, Diogo, Pilar,
FPT, Pellegrino}.

Connected to the notion of lineability, for a subspace $\mathcal{S} $ of
$\mathbb{K}^{\mathbb{N}} $, we say that a subset $A $ of a Hausdorff
topological vector space $X $ is $[ (u_{n})_{n=1}^{\infty}, \mathcal{S} ]
$-lineable in $X $ if, for each sequence $(c_{n})_{n=1}^{\infty}
\in\mathcal{S} $, the series $\sum_{n=1}^{\infty} c_{n} u_{n} $ converges in
$X $ to a vector in $A \cup\{ 0 \} $. Moreover, $A $ is $[ \mathcal{S} ]
$-lineable in $X $ if it is $[ (u_{n})_{n=1}^{\infty}, \mathcal{S} ]
$-lineable for some sequence $(u_{n})_{n=1}^{\infty} $ of linearly independent
elements in $X $.

This concept was introduced in \cite{S lineability} by Bernal-Gonz\'{a}lez et
al$.$ and originally coined by V. Gurariy and R. Aron during a Non-Linear
Analysis Seminar at Kent State University. As far as we know, this notion was
initially inspired by a result of Levine and Milman from 1940 \cite{Levine e
Milman}, created to address the lack of such a concept regarding convergence.
Regarding the result of Levine and Milman, they demonstrate that the subset of
$C[0,1]$ consisting of all functions of bounded variation does not contain a
closed subspace of infinite dimension.

\bigskip The results established in this paper are:

\begin{itemize}
\item \textbf{(Main Result)} For an $F$-space $E$, if $W$ is a subspace of $E$
that contains a non-minimal closed subspace, then the set $E \setminus W$ is
spaceable if and only if it is $[ (x_{n})_{n=1}^{\infty}, \mathcal{S}]
$-lineable for some closed subspace $\mathcal{S}$ of $\ell_{\infty}$
containing $c_{0}$ and some $\mathcal{S}$-independent sequence $(x_{n}%
)_{n=1}^{\infty}$ of elements of $E$.\newline

\item For a Banach space $E$, if $F$ is a closed subspace of $E$, then the
following conditions are equivalent:

\begin{enumerate}
\item $E \setminus F$ is $[ (x_{n})_{n=1}^{\infty}, \mathcal{S}] $-lineable
for some closed subspace $\mathcal{S}$ of $\ell_{\infty}$ containing $c_{0}$
and some $\mathcal{S}$-independent sequence $(x_{n})_{n=1}^{\infty}$ of
elements of $E$,

\item $E \setminus F$ is spaceable,

\item $F$ has infinite codimension.\newline
\end{enumerate}

\item For an $F$-space $E$, if $W$ is a subspace of $E$ that contains a
non-minimal closed subspace, then the following conditions are equivalent:

\begin{enumerate}
\item $E \setminus W$ is spaceable,

\item $E \setminus W$ is $[ (x_{n})_{n=1}^{\infty}, \mathcal{S}] $-lineable
for some closed subspace $\mathcal{S}$ of $\ell_{\infty}$ containing $c_{0}$
and some $\mathcal{S}$-independent sequence $(x_{n})_{n=1}^{\infty}$ of
elements of $E$,

\item $E \setminus W$ is $\left(  n, \mathfrak{c} \right)  $-spaceable for
each $n \in\mathbb{N}$.
\end{enumerate}
\end{itemize}

\section{Preliminaries\label{section 2}}

In this paper, we will characterize some families of complements of subspaces
(not necessarily closed) through the notion of $\left[  \mathcal{S}\right]
$-lineability. To this end, let us begin with some preliminary notions and results.

\begin{definition}
\cite[p. 376]{Drew} We say tha a sequence $\left(  u_{n}\right)
_{n=1}^{\infty}$ of elements of a topological vector space $E$ is
topologically linearly independent, if for each sequence $\left(
c_{n}\right)  _{n=1}^{\infty}\in\mathbb{K}^{\mathbb{N}}$ with $\sum
_{n=1}^{\infty}c_{n}u_{n}=0$, we have $\left(  c_{n}\right)  _{n=1}^{\infty
}=0$.
\end{definition}

Based on this definition, if $\mathcal{S}\neq\left\{  0\right\}  $ is a
subspace of $\mathbb{K}^{\mathbb{N}}$ then we will say that a sequence
$\left(  u_{n}\right)  _{n=1}^{\infty}$ of elements of a topological vector
space $E$ is $\mathcal{S}$-topologically linearly independent in\textbf{ }$E$
(or $\mathcal{S}$-independent) if for each sequence $\left(  c_{n}\right)
_{n=1}^{\infty}\in\mathcal{S}$ with $\sum_{n=1}^{\infty}c_{n}u_{n}=0$, we have
$\left(  c_{n}\right)  _{n=1}^{\infty}=0$.

\bigskip

In connection with this notion, we present the following classic definition.

\begin{definition}
A sequence $\left(  e_{k}\right)  _{k=1}^{\infty}$ in a metrizable topological
vector space $E$ is said to be basic whenever every $x\in\operatorname*{span}%
\left\{  e_{k}\right\}  _{k=1}^{\infty}$, the closed linear span of the
$x_{j}$'s, can be uniquely represented as a convergent series $x=\sum
_{k=1}^{\infty}a_{k}e_{k}$. Furthermore, as in \cite{Kalton}, a basic sequence
is called \emph{regular} if it is bounded away from zero, that is, if it lies
entirely outside some neighborhood of zero.
\end{definition}

From the perspective of minimal subspaces, as discussed in the introduction,
Kalton and Shapiro in \cite{KShapiro} proved the following result:

\begin{proposition}
If $E$ is an $F$-space, then the following are equivalent:

\begin{enumerate}
\item[$\left(  1\right)  $] $E$ is non-minimal,

\item[$\left(  2\right)  $] $E$ contains a regular basic sequence.
\end{enumerate}
\end{proposition}

\bigskip

The following Lemma is crucial for the next results.

\begin{lemma}
\label{Lemadacompacidade}\cite[Lemma 2]{Kalton 2} For $F$-spaces $Z$ and $E$,
if $K\colon Z\rightarrow E$ is a compact operator and $\mathcal{I}\colon
Z\rightarrow E$ is the identity embendding then the operator $\mathcal{I}%
+K\colon Z\rightarrow E$ has closed range.
\end{lemma}

Based on these notions, Drewnowski in \cite{Drew} demonstrated the following result:

\begin{theorem}
\label{TheoDrew}$\left(  \text{\cite[Theorem 3.3]{Drew}}\right)  $ For an
$F$-space $E$, if $F$ is a closed subspace non-minimal of $E$ and $\dim\left(
E/F\right)  =\infty$, then there exists a closed non-minimal subspace
$\emph{F}$ of $E$ such that $F\cap\emph{F}=\left\{  0\right\}  $ and
$F+\emph{F}$ is not closed.
\end{theorem}

To the best of our knowledge, there is no criterion in the literature that
establishes spaceability for the complement of a subspace $W$ (not necessarily
closed) in an $F$-space $E$.

\section{The Main Result}

The next result is a characterization for spaceability in the context of
complements of subspaces that are not necessarily closed, which was inspired
by \ref{TheoDrew}.

\begin{theorem}
\label{Teo3.6} For an $F$-space $E$, if $W$ be a subspace of $E$ that contains
a non-minimal closed subspace, then the set $E\setminus W$ is spaceable if and
only if it is $[(x_{n})_{n=1}^{\infty},\mathcal{S}]$-lineable for some closed
subspace $\mathcal{S}$ of $\ell_{\infty}$ containing $c_{0}$ and some
$\mathcal{S}$-independent sequence $(x_{n})_{n=1}^{\infty}$ of elements of $E$.
\end{theorem}

\begin{proof}
Assume first that $E\setminus W$ is spaceable and let $F$ be an
infinite-dimensional closed subspace of $E$ such that%
\[
F\cap W=\left\{  0\right\}  \text{.}%
\]
Considering an $\ell_{\infty}$-independent sequence $(x_{n})_{n=1}^{\infty}$
of elements of $F$ such that
\[
\sum_{n=1}^{\infty}\Vert x_{n}\Vert_{E}<\infty\text{,}%
\]
let us show that $E\setminus W$ is $[(x_{n})_{n=1}^{\infty},\mathcal{S}%
]$-lineable, for $\mathcal{S}:=\ell_{\infty}$. For this, consider the
following linear operator:
\[
\mathcal{A}:\ell_{\infty}\rightarrow F,\quad(t_{n})_{n=1}^{\infty}\mapsto
\sum_{n=1}^{\infty}t_{n}x_{n}.
\]
Since $\sum_{n=1}^{\infty}\Vert x_{n}\Vert_{E}<\infty$, the operator
$\mathcal{A}$ is well-defined. Moreover, due to the fact that $(x_{n}%
)_{n=1}^{\infty}$ is $\ell_{\infty}$-independent, $\mathcal{A}$ is also
injective. This allows infer that%
\[
\mathcal{A}\left(  \ell_{\infty}\right)  \subseteq F\subseteq(E\setminus
W)\cup\{0\}\text{,}%
\]
showing that $E\setminus W$ is $[(x_{n})_{n=1}^{\infty},\mathcal{S}%
]$-lineable, as desired.

For the converse, assume that $E\setminus W$ is $[(x_{n})_{n=1}^{\infty
},\mathcal{S}]$-lineable for some closed subspace $\mathcal{S}$ of
$\ell_{\infty}$ containing $c_{0}$ and some $\mathcal{S}$-independent sequence
$(x_{n})_{n=1}^{\infty}$ of elements of $E$. Consider the following operator:
\begin{equation}
\mathcal{O}:\mathcal{S}\rightarrow E,\quad(t_{n})_{n=1}^{\infty}\mapsto
\sum_{n=1}^{\infty}t_{n}x_{n}.\label{TheoperatorO}%
\end{equation}
Given that $E\setminus W$ is $[(x_{n})_{n=1}^{\infty},\mathcal{S}]$-lineable,
we have:%
\begin{equation}
\mathcal{O}(\mathcal{S})\subseteq(E\setminus W)\cup\{0\}\text{.}%
\label{OScontidoemX}%
\end{equation}
Notice that%
\[
\dim\mathcal{O}(\mathcal{S})=\dim\mathcal{S}=\infty\text{ }(\mathcal{O}\text{
is injective}).
\]
Considering $\overline{\text{span}}\{y_{n}\}_{n=1}^{\infty}$ where $\left(
y_{n}\right)  _{n=1}^{\infty}$ is a regular basic sequence, $\left(  W\text{
\textbf{has a non-minimal closed subspace}}\right)  $ this allows infer that
the set $E\setminus\overline{\text{span}}\{y_{n}\}_{n=1}^{\infty}$ has
infinite-codimension. That is,%
\[
\dim\left(  E/\overline{\text{span}}\{y_{n}\}_{n=1}^{\infty}\right)
=\infty\text{.}%
\]

Without loss of generality, assume that $(x_{n})_{n=1}^{\infty}$ is such that%
\[
\sum_{n=1}^{\infty}\Vert x_{n}\Vert_{E}<\infty
\]
and%
\[
\sum_{n=1}^{\infty}a_{n}Q\left(  x_{n}\right)  =0\Rightarrow a_{n}=0\text{
whenever }\left(  a_{n}\right)  _{n=1}^{\infty}\in\ell_{\infty}%
\]
where $Q\colon E\rightarrow E/\overline{\text{span}}\{y_{n}\}_{n=1}^{\infty}$
is the quotient map.

Let:
\[
\mathcal{L}:\overline{\text{span}}\{y_{n}\}_{n=1}^{\infty}\rightarrow
\mathcal{S},\quad\mathcal{L}(y)=(f_{j}(y))_{j=1}^{\infty},
\]
where $f_{j}:\overline{\text{span}}\{y_{n}\}_{n=1}^{\infty}\rightarrow
\mathbb{K}$ is such that%
\[
f_{j}(y_{n})=\left\{
\begin{array}
[c]{ccc}%
0\text{,} & \text{if} & n\neq j\text{,}\\
1\text{,} & \text{if} & n=j\text{.}%
\end{array}
\right.
\]

The map $\mathcal{L}$ is well-defined since $(y_{n})_{n=1}^{\infty}$ is
regular and $c_{0}\subseteq\mathcal{S}$. Furthermore, is is continuous due to
the \textbf{Closed Graph Theorem}. Regarding the operator $\mathcal{O}$
considered in $($\ref{TheoperatorO}$)$, it is also continuous. To see this, it
suffices to consider the sequence of continuous mappings $(\mathcal{O}%
_{N})_{N=1}^{\infty}$ defined by
\[
\mathcal{O}_{N}:\ell_{\infty}\rightarrow E,\quad\mathcal{O}_{N}((t_{k}%
)_{k=1}^{\infty})=\sum_{k=1}^{N}t_{k}x_{k},
\]
and employ the version of the \textbf{Banach-Steinhaus} in \cite[Theorem
2.8]{Rudin}.

Fix $y=\sum_{n=1}^{\infty}a_{n}y_{n}$ in $\overline{\text{span}}%
\{y_{n}\}_{n=1}^{\infty}$. Due to the fact that
\[
(\mathcal{O}\circ\mathcal{L})\left(  y\right)  =(\mathcal{O}\circ
\mathcal{L})\left(  \sum_{n=1}^{\infty}a_{n}y_{n}\right)  =\mathcal{O}\left(
\sum_{n=1}^{\infty}a_{n}\mathcal{L}(y_{n})\right)  =\mathcal{O}((a_{n}%
)_{n=1}^{\infty})=\sum_{n=1}^{\infty}a_{n}x_{n}%
\]
and
\[
f_{j}\left(  y\right)  =f_{j}\left(  \sum_{n=1}^{\infty}a_{n}y_{n}\right)
=a_{j},\quad\text{for each }j\in\mathbb{N}%
\]
we obtain
\[
(\mathcal{O}\circ\mathcal{L})(y)=\sum_{n=1}^{\infty}f_{n}(y)x_{n}.
\]

Since $y$ is arbitrary, o operador $\mathcal{O}\circ\mathcal{L}$ is
well-defined and it is determined by sequences $\left(  f_{n}\right)
_{n=1}^{\infty}$ and $\left(  x_{n}\right)  _{n=1}^{\infty}$. Moreover,
according to the proof of Proposition 3.3 in \cite{D}, this is a operator
compact. Hence, if we consider the identity embedding $\mathcal{I}%
:\overline{\text{span}}\{y_{k}\}_{k=1}^{\infty}\rightarrow E$, and invoke the
Lemma \ref{Lemadacompacidade} , we conclude that the operator $\mathcal{I}%
+\mathcal{O}\circ\mathcal{L}$ is such that%
\[
\left(  \mathcal{I}+\mathcal{O}\circ\mathcal{L}\right)  \left(  \overline
{\text{span}}\{y_{k}\}_{k=1}^{\infty}\right)
\]
is closed.

Furthermore, since:

1. $(\mathcal{O}\circ\mathcal{L})(y)=0\Rightarrow\sum_{n=1}^{\infty}%
f_{n}(y)x_{n}=0$,

2. $(x_{n})_{n=1}^{\infty}$ is $\ell_{\infty}$-independent,

3. $(f_{n}(y))_{n=1}^{\infty}\in\mathcal{S}$, and

4. $\mathcal{S}\subseteq\mathcal{\ell}_{\infty}$,

we have:
\[
f_{n}(y)=0,\quad\text{for each }n\in\mathbb{N}.
\]
This allows us to conclude that
\[
y=0.
\]
Hence, $\mathcal{O}\circ\mathcal{L}$ is injective and consequently the
operator $\mathcal{I}+\mathcal{O}\circ\mathcal{L}$ is an isomorphism over its range.

We claim that
\[
(\mathcal{I}+\mathcal{O}\circ\mathcal{L})(\overline{\text{span}}%
\{y_{n}\}_{n=1}^{\infty})\cap W=\{0\}.
\]

Indeed, let $v\in(\mathcal{I}+\mathcal{O}\circ\mathcal{L})(\overline
{\text{span}}\{y_{n}\}_{n=1}^{\infty})\cap W$. Due to the fact that
$v\in(\mathcal{I}+\mathcal{O}\circ\mathcal{L})(\overline{\text{span}}%
\{y_{n}\}_{n=1}^{\infty})$, there exists $u$ in $\overline{\text{span}}%
\{y_{k}\}_{k=1}^{\infty}$ such that
\[
v=(\mathcal{I}+\mathcal{O}\circ\mathcal{L})(u).
\]
Since
\[
v=(\mathcal{I}+\mathcal{O}\circ\mathcal{L})(u)=u+(\mathcal{O}\circ
\mathcal{L})(u),
\]
we have
\[
(\mathcal{O}\circ\mathcal{L})(u)=v-u\in\left[  W+\overline{\text{span}}%
\{y_{n}\}_{n=1}^{\infty}\right]  \subseteq W.
\]
That is,%
\[
(\mathcal{O}\circ\mathcal{L})(u)\in W.
\]
On the other hand, due to the fact that
\[
(\mathcal{O}\circ\mathcal{L})(u)=\sum_{n=1}^{\infty}f_{n}(u)x_{n}%
\in\mathcal{O}(\mathcal{S}),
\]
we can invoke $\left(  \text{\ref{OScontidoemX}}\right)  $ and conclude that
\[
v-u=(\mathcal{O}\circ\mathcal{L})(u)=0.
\]
Therefore,
\[
v=u=0\text{ }\left(  \mathcal{O}\circ\mathcal{L}\text{ is injective}\right)
\text{,}%
\]
and consequently,%
\[
(\mathcal{I}+\mathcal{O}\circ\mathcal{L})(\overline{\text{span}}%
\{y_{n}\}_{n=1}^{\infty})\cap W=\{0\}
\]
as desired.

\bigskip

Finally, since $(\mathcal{I}+\mathcal{O}\circ\mathcal{L})(y_{n})=y_{n}%
+(\mathcal{O}\circ\mathcal{L})(y_{n})=y_{n}+x_{n}$, we have that the sequence
$(y_{n}+x_{n})_{n=1}^{\infty}$ is basic, and this allows us to conclude that%
\[
\dim(\mathcal{I}+\mathcal{O}\circ\mathcal{L})(\overline{\text{span}}%
\{y_{n}\}_{n=1}^{\infty})=\infty.
\]
Therefore, $E\setminus W$ is spaceable, completing the proof of this result.
\end{proof}

\section{Some Consequences}

\begin{corollary}
For a Banach space $E$, if $F$ be an infinite-dimensional closed subspace of
$E$, then the following conditions are equivalent:

\begin{enumerate}
\item $E\setminus F$ is $[(x_{n})_{n=1}^{\infty},\mathcal{S}]$-lineable for
some closed subspace $\mathcal{S}$ of $\ell_{\infty}$ containing $c_{0}$ and
some $\mathcal{S}$-independent sequence $(x_{n})_{n=1}^{\infty}$ of elements
of $E$,

\item $E\setminus F$ is spaceable,

\item $F$ has infinite codimension.
\end{enumerate}
\end{corollary}

\begin{corollary}
For an $F$-space $E$, if $W$ be a subspace of $E$ that contains a non-minimal
closed subspace,then the following conditions are equivalent:

\begin{enumerate}
\item $E\setminus W$ is spaceable,

\item $E\setminus W$ is $[(x_{n})_{n=1}^{\infty},\mathcal{S}]$-lineable for
some closed subspace $\mathcal{S}$ of $\ell_{\infty}$ containing $c_{0}$ and
some $\mathcal{S}$-independent sequence $(x_{n})_{n=1}^{\infty}$ of elements
of $E$,

\item $E\setminus W$ is $\left(  n,\mathfrak{c}\right)  $-spaceable for each
$n\in\mathbb{N}$.
\end{enumerate}
\end{corollary}

\section{Acknowledgements}

We would like to thank Professors Dr. Anselmo R. Jr. and Dr. Fernando V. J. C.
for their valuable contributions during the development of this paper.

\end{document}